\documentclass[12pt,a4paper]{amsart}

\usepackage[english]{babel}
\usepackage{amsthm,amsmath,amsfonts,amssymb,amscd,mathrsfs,comment}
\usepackage[symbol]{footmisc}
\usepackage{tikz}
\usepackage[colorlinks=true,linkcolor=blue,citecolor=blue,urlcolor=blue]{hyperref}
\usepackage{csquotes}
\usepackage{doi}

\newcommand{\PSL}{\operatorname{PSL}}

\theoremstyle{plain}
\newtheorem{theorem}{Theorem}[section]

\newtheorem{lemma}[theorem]{Lemma}

\newtheorem{proposition}[theorem]{Proposition}
\newtheorem{corollary}[theorem]{Corollary}

\newtheorem{question}[theorem]{Question}

\theoremstyle{definition}

\newtheorem{example}[theorem]{Example}

\begin{document}

\title[Brauer--Fowler Bounds for Elements of Odd Prime Order]{Brauer--Fowler Bounds for Elements of Odd Prime Order}

\author{Alessandro DIOGUARDI BURGIO, Kıvanç ERSOY, Edoardo SALATI}

\address{Università degli Studi di Palermo\\
Dipartimento di Matematica e Informatica\\
via Archirafi 34, 90123\\
Palermo\\
Italy}
\email{alessandro.dioguardiburgio@unipa.it}

\address{Freie Universit\"at Berlin\\
Fachbereich Mathematik und Informatik\\
Arnimallee 7\\
14195 Berlin\\
Germany}
\email{ersoy@zedat.fu-berlin.de}

\address{RPTU University Kaiserslautern-Landau\\
Department of Mathematics\\
Gottlieb-Daimler-Stra\ss e 48\\
67663 Kaiserslautern\\
Germany}
\email{edoardo.salati@rptu.de}

\begin{abstract}
The Brauer--Fowler theorem bounds the order of a finite simple group
in terms of the order of the centralizer of an involution. Hartley
proved an automorphism version, bounding the order of a finite simple
group in terms of the order of an automorphism and the order of its
fixed-point subgroup.

Strunkov asked whether, in the involution case, the order of the
centralizer can be replaced by the number of involutions commuting
with the given involution; this was recently answered affirmatively by
Skresanov. Motivated by this question, and in contrast with Hartley's
use of the order of a fixed-point subgroup, we study what can be
deduced from counting elements of prescribed prime order inside such
a subgroup.

We show that the direct analogue of Skresanov's result fails for elements of odd prime order.
We prove that if \(G\) is a finite simple group,
\(x\in G\) has odd prime order \(p\), \(C_G(x)\) contains at most \(k\)
elements of order \(p\), and the exponent of \(C_G(x)\) is at most
\(e\), then \(|G|\) is bounded in terms of \(k\) and \(e\). 
\end{abstract}
\maketitle
\section{Introduction}

The Brauer--Fowler theorem \cite[Theorem (2D)]{BrauerFowler}, a
cornerstone in the study of the classification of finite simple groups,
shows that the order of a finite simple group is bounded in terms of
the order of the centralizer of an involution.

A key feature of Brauer and Fowler's proof is the use of counting
arguments involving involutions, in particular involutions centralizing
a given involution. Indeed, one can restate their result as follows: if
\(G\) is a finite simple group with \(m\) involutions and if
\(n=|G|/m\), then \cite[Corollary (2I)]{BrauerFowler}
\[
        |G|\leq \left[\frac{n(n+1)}{2}\right]!.
\]

After the classification of finite simple groups, several variants and
refinements of the classical Brauer--Fowler theorem were obtained. We
first recall an automorphism version due to B.~Hartley, proved using
the classification of finite simple groups.

\begin{theorem}\cite[Theorem A']{Hartley1992}
There exists an integer-valued function
\[
        f:\mathbb N\times\mathbb N\longrightarrow\mathbb N
\]
such that if \(\alpha\) is an automorphism of a non-abelian finite
simple group \(G\) with \(|\alpha|=m\) and \(|C_G(\alpha)|\leq k\),
then
\[
        |G|\leq f(m,k).
\]
\end{theorem}

Hartley's result has the following consequences.

\begin{theorem}\cite[Theorem A]{Hartley1992}
There exists an integer-valued function
\[
        f:\mathbb N\times\mathbb N\longrightarrow\mathbb N
\]
such that if \(x\) is an element of a finite group \(G\) with
\(|x|=m\) and \(|C_G(x)|\leq k\), then \(G\) has a solvable normal
subgroup of index dividing \(f(m,k)\).
\end{theorem}

Hartley also obtained the following locally finite consequence.

\begin{corollary}\cite[Corollary A1]{Hartley1992}
Let \(G\) be a locally finite group with an element whose centralizer
is finite. Then \(G\) has a locally solvable normal subgroup of finite
index.
\end{corollary}

Further variants of the Brauer--Fowler theorem for finite groups were
obtained by Guralnick and Robinson in \cite{GuralnickRobinson2020} and by Jabara in \cite{jab01}.

In \cite{Strunkov1985}, Strunkov proved an analogue of the
Brauer--Fowler theorem showing that, in certain specific situations,
the relevant data are not the involution centralizers themselves, but
rather the involutions contained in them. In particular, his result
gives finiteness statements for finite simple groups in terms of the
numbers of involutions occurring in the centralizers of non-conjugate
involutions, as well as related results for infinite periodic groups.
He later asked the following more general question.

\begin{question}\cite[Question 11.96 by Strunkov]{KourovkaNotebook}
Is it true that, for a given number \(n\), there exist only finitely many
finite simple groups each of which contains an involution which commutes
with at most \(n\) involutions of the group?
\end{question}

Recently, S.~Skresanov answered this question affirmatively. More
precisely, he proved the following result, using the classification of
finite simple groups.

\begin{theorem}\cite[Theorem 1]{Skresanov2024}\label{skresanov}
Let \(G\) be a simple locally finite group and let \(t\) be an involution
in \(G\). If at most \(n\) involutions of \(G\) commute with \(t\), then
\(G\) is finite and \(|G|\) is bounded in terms of \(n\) alone.
\end{theorem}

Theorem~\ref{skresanov} suggests the corresponding problem for
elements of odd prime order. Namely, if \(x\in G\) has order \(p\),
where \(p\) is an odd prime, does a bound on the number of elements of order
\(p\) contained in \(C_G(x)\) force a bound on \(|G|\)? In
Section~\ref{section2}, we show that this direct analogue is false.

The counterexamples in Section~\ref{section2} show that, for elements
of odd prime order, bounding only the number of elements of order \(p\)
in the centralizer is not sufficient. Our main result shows that
boundedness is recovered after imposing an additional restriction on the
exponent of the centralizer. This condition is weaker than bounding the
order of the centralizer, and so the result is not a direct consequence
of Hartley's theorem.

\begin{theorem}\label{mainthm}
There exists an integer-valued function
\[
        f:\mathbb N^2\to\mathbb N
\]
such that the following holds. Let \(G\) be a finite simple
group and let \(x\in G\) be an element of odd prime order \(p\). Suppose
that \(C_G(x)\) contains at most \(k\) elements of order \(p\), and that
\[
        \exp(C_G(x))\leq e.
\]
Then
\[
        |G|\leq f(k,e).
\]
\end{theorem}
\section{Variants of the problem and counterexamples}\label{section2}

We discuss several natural variants and show that they have negative
answers. The first one is the direct analogue of Skresanov's theorem for
elements of odd prime order.

\begin{question}\label{op1}
Does there exist a function
\[
        f:\mathbb N\longrightarrow\mathbb N
\]
such that, if \(G\) is a finite simple group with an element \(x\) of
order \(p\), for some odd prime \(p\), and if \(C_G(x)\) contains at
most \(k\) elements of order \(p\), then
\[
        |G|\leq f(k)?
\]
\end{question}

The corresponding locally finite question is the following.

\begin{question}\label{simplin}
Let \(G\) be a simple locally finite group with an automorphism
\(\alpha\) of order \(p\), where \(p\) is an odd prime. Suppose that
\(C_G(\alpha)\) contains only finitely many elements of order \(p\).
Must \(G\) be finite?
\end{question}

The following example gives a negative answer to both questions.
\begin{example}\label{example:psl-counter-example}
For \(i\geq 1\), set
\[
        q_i=3^{2^{2i}}
        \qquad\text{and}\qquad
        G_i=\PSL_2(q_i).
\]
Since \(\mathbb F_{q_i}\subseteq \mathbb F_{q_{i+1}}\), the groups
\(G_i\) form an increasing chain of finite simple groups.

We show that \(41\) divides \(|G_i|\) exactly once for every \(i\geq 1\).
For \(i=1\), this follows from
\[
        q_1=81,\qquad q_1+1=82=2\cdot 41,\qquad q_1-1=80.
\]
Let \(i\geq 2\). Write \(2^{2i}=8m_i\). Then \(41\nmid m_i\), and
\[
        3^8=1+41\cdot 160.
\]
Thus
\[
        q_i-1=(3^8)^{m_i}-1
        \equiv 41\cdot 160m_i \pmod {41^2}.
\]
Hence \(41\mid q_i-1\), but \(41^2\nmid q_i-1\). Moreover,
\(q_i\equiv 1\pmod {41}\), so \(41\nmid q_i+1\). Since
\[
        |\PSL_2(q_i)|=\frac{q_i(q_i-1)(q_i+1)}{2},
\]
the claim follows.

For \(i=1\), an element of order \(41\) lies in a non-split maximal
torus of order \((q_1+1)/2=41\). For \(i\geq 2\), an element of order
\(41\) lies in a split maximal torus, whose \(41\)-part has order
\(41\). In either case, the centralizer of such a semisimple element is
the corresponding maximal torus. Hence, for each \(i\), there exists
\(x_i\in G_i\) of order \(41\) such that \(C_{G_i}(x_i)\) contains
exactly \(40\) elements of order \(41\). Since \(|G_i|\to\infty\), the
answer to Question~\ref{op1} is negative.

Now choose \(x\in G_1\) of order \(41\), and set
\[
        G=\bigcup_{i\geq 1}G_i.
\]
Then \(G\) is an infinite simple locally finite group. Under the natural
embeddings \(G_1\leq G_i\), the element \(x\) remains semisimple and its
centralizer in \(G_i\) has \(41\)-part of order \(41\). Therefore
\(C_G(x)\) contains exactly the \(40\) non-trivial elements of
\(\langle x\rangle\) of order \(41\). Thus Question~\ref{simplin} also
has a negative answer.
\end{example}

We next consider elements of odd prime order whose centralizers contain
only a bounded number of involutions.

\begin{question}\label{op3}
Does there exist a function
\[
        h:\mathbb N\longrightarrow\mathbb N
\]
such that, if \(G\) is a finite group with an element \(x\) of odd
prime order \(p\), and if \(C_G(x)\) contains at most \(k\) involutions,
then \(G\) has a solvable normal subgroup of index bounded by \(h(k)\)?
\end{question}

\begin{example}\label{example:regular-unipotent}
Let \(n\geq 4\) be even, and choose an odd prime \(p_n>n\). Set
\[
        G_n=\mathrm{SL}_n(p_n).
\]
Let \(x_n\in G_n\) be a regular unipotent element, that is, an element whose Jordan form consists of a single Jordan block. Since \(p_n>n\), the
element \(x_n\) has order \(p_n\).

We claim that \(C_{G_n}(x_n)\) contains exactly one involution. Write
\[x_n=I+N\] where \(N\) is a nilpotent Jordan block of size \(n\). It
is standard that
\[
        C_{GL_n(p_n)}(x_n)
        =
        \left\{
        a_0I+a_1N+\cdots+a_{n-1}N^{n-1}
        : a_0\neq 0
        \right\};
\]
see \cite[1.5, Example (d)(3)]{spst}. In particular, every semisimple
element of \(C_{\mathrm{GL}_n(p_n)}(x_n)\) is scalar.

Now let \(y\in C_{G_n}(x_n)\) be an involution. Since \(p_n\) is odd,
\(y\) is semisimple, and hence scalar. Thus \(y=\pm I\). Since \(n\)
is even, \(-I\in \mathrm{SL}_n(p_n)\). Therefore \(C_{G_n}(x_n)\) contains
exactly one involution, namely \(-I\).

On the other hand,
\[
        G_n/Z(G_n)\cong \PSL_n(p_n)
\]
is non-abelian simple. Hence every solvable normal subgroup of \(G_n\)
is contained in \(Z(G_n)\). Therefore the index of any solvable normal
subgroup of \(G_n\) is at least
\[
        |\PSL_n(p_n)|,
\]
which tends to infinity as \(n\to\infty\). Thus no function \(h\) as in
Question~\ref{op3} can exist.
\end{example}

We next return to the more natural finite-group analogue in which one
counts elements of the same prime order as \(x\). In the spirit of
Hartley's finite-group theorem, we allow the bound to depend also on
that prime.

\begin{question}\label{question:finite-index-pk}
Does there exist a function
\[
        h:\mathbb N\times\mathbb N\longrightarrow\mathbb N
\]
such that, if \(G\) is a finite group with an element \(x\) of order
\(p\), and if \(C_G(x)\) contains at most \(k\) elements of order \(p\),
then \(G\) has a solvable normal subgroup of index bounded by
\(h(p,k)\)?
\end{question}

\begin{example}\label{example:psl2-pk}
Let
\[
        G_i=\PSL_2(3^{2^i}),\qquad i\geq 2.
\]
For each \(i\), the group \(G_i\) contains an element \(x_i\) of order
\(41\). Moreover, \(C_{G_i}(x_i)\) is a cyclic maximal torus whose
\(41\)-part has order \(41\). Hence \(C_{G_i}(x_i)\) contains exactly
\(40\) elements of order \(41\).

Since the groups \(G_i\) are finite simple groups and
\[
        |G_i|\to\infty,
\]
they cannot have solvable normal subgroups of index bounded in terms
of \(41\) and \(40\). Thus Question~\ref{question:finite-index-pk}
has a negative answer.
\end{example}

The preceding examples show that neither the direct simple-group
analogue nor the Hartley-type finite-group analogue follows from
bounding only the number of \(p\)-elements in the centralizer. We next
record that adding a bound on the number of involutions in the same
centralizer is still not sufficient.

\begin{question}\label{op2}
Does there exist a function
\[
        f:\mathbb N^3\longrightarrow\mathbb N
\]
such that, if \(G\) is a finite simple group with an automorphism
\(\alpha\) of odd prime order \(p\), and \(C_G(\alpha)\) contains at most
\(k\) elements of order \(p\) and at most \(m\) involutions, then
\[
        |G|\leq f(p,k,m)?
\]
\end{question}

\begin{example}\label{example:psl2-involutions}
Let
\[
        G_i=\PSL_2(3^{2^i}),\qquad i\geq 3.
\]
Then \(41\mid 3^{2^i}-1\), and the \(41\)-part of \(3^{2^i}-1\) has
order exactly \(41\). Hence \(G_i\) contains an element \(x_i\) of
order \(41\) whose centralizer is a cyclic split torus. This centralizer
contains exactly \(40\) elements of order \(41\).

Moreover, since \(i\geq 3\), the order of this torus is even and since it is cyclic,
it contains a unique involution. Let \(\alpha_i\) be the inner
automorphism of \(G_i\) induced by \(x_i\). Then
\[
        C_{G_i}(\alpha_i)=C_{G_i}(x_i)
\]
contains exactly \(40\) elements of order \(41\) and exactly one
involution. But
\[
        |G_i|\to\infty.
\]
Therefore no function \(f\) as in Question~\ref{op2} can exist.
\end{example}

The preceding example shows that bounding both the number of
\(p\)-elements and the number of involutions in the fixed-point subgroup
of an automorphism of order \(p\) still does not force a bound on the
order of the group. The following related example, in characteristic \(2\), shows
that the same phenomenon can occur even more sharply: the centralizer
may contain exactly \(p-1\) elements of order \(p\) and no involutions
at all.

We thank Professor Khukhro for pointing out the following example.

\begin{example}\cite[Example 1]{khumaz}\label{example:khukhro}
Let \(p\) be an odd prime, and set
\[
        G_n=\mathrm{SL}_2(2^{n(p-1)}).
\]
Since \(2^{p-1}\equiv 1\pmod p\), the group \(G_n\) contains a
semisimple element \(x_n\) of order \(p\). Let \(\varphi_n\) be the
inner automorphism induced by \(x_n\). Then \(C_{G_n}(\varphi_n)\)
is a cyclic maximal torus of order \(2^{n(p-1)}-1\). In particular,
\(C_{G_n}(\varphi_n)\) contains exactly \(p-1\) elements of order
\(p\).

Since \(2^{n(p-1)}-1\) is odd, this centralizer contains no
involutions. Nevertheless,
\[
        |G_n|\to\infty
\]
as \(n\to\infty\).
\end{example}

For a periodic group \(H\), let \(\pi(H)\) be the set of primes \(r\)
for which \(H\) contains an element of order \(r\).

\begin{question}\label{question:bounded-pi}
Does there exist a function
\[
        f:\mathbb N^2\longrightarrow\mathbb N
\]
such that the following holds? Let \(G\) be a finite non-abelian simple
group, and let \(x\in G\) be an element of odd prime order \(p\). Suppose
that \(C_G(x)\) contains at most \(k\) elements of order \(p\), and that
\[
        |\pi(C_G(x))|\leq m.
\]
Must
\[
        |G|\leq f(k,m)?
\]
\end{question}

\begin{example}\label{example:bounded-pi-counterexample}
Fix an odd prime \(p\). The pair
\[
        L_1(n)=2p(n+1)+1,
        \qquad
        L_2(n)=n+1
\]
is admissible. By the generalized Chen theorem for two admissible linear forms
\cite[Theorem 6.4]{MatomakiShao2017}, there exist infinitely many
integers \(n_i\) such that 
\[
        q_i:=2p(n_i+1)+1
\]
is prime and \(n_i+1\) has at most two prime divisors.

Set
\[
        G_i=\PSL_2(q_i).
\]
Since
\[
        p\mid \frac{q_i-1}{2},
\]
the group \(G_i\) contains an element \(x_i\) of order \(p\) in a split
maximal torus. This element is regular semisimple, and hence
\(C_{G_i}(x_i)\) is the corresponding cyclic split torus of order
\[
        \frac{q_i-1}{2}=p(n_i+1).
\]
Thus \(C_{G_i}(x_i)\) contains exactly \(p-1\) elements of order \(p\).
Moreover, since \(n_i+1\) has at most two prime divisors, we have
\[
        |\pi(C_{G_i}(x_i))|\leq 3.
\]
However,
\[
        |G_i|=\frac{q_i(q_i^2-1)}{2}\to\infty.
\]
Thus the answer to Question~\ref{question:bounded-pi} is negative.
\end{example}

\section{Auxiliary results}

In this section we collect the auxiliary results needed in the proof of
the main theorem. We begin by proving a more general result for
alternating groups.

\begin{lemma}\label{lemmaalt}
There exists a function
\[
        f:\mathbb N^2\to \mathbb N
\]
such that the following holds:

Let \(n\geq 7\), let \(G=A_n\), and let
\(\alpha\in \operatorname{Aut}(G)\) be an automorphism of prime order
\(p\). Suppose that \(C_G(\alpha)\) contains at most \(k\) elements of
order \(p\). Then
\[
        n\leq f(p,k).
\]
\end{lemma}

\begin{proof}
Since \(n\geq 7\), by \cite[Theorem 8.2A and Lemma 8.2B]{DiMo},
\[
        \operatorname{Aut}(A_n)\cong S_n .
\]
Let \(\sigma\in S_n\) induce \(\alpha\) by conjugation. Since
\(\alpha\) has order \(p\), so does \(\sigma\). Thus \(\sigma\) is a
product of \(r\geq 1\) disjoint \(p\)-cycles and fixes \(t=n-rp\)
points.

It is standard that
\[
        C_{S_n}(\sigma)
        \cong
        (C_p^r\rtimes S_r)\times S_t ,
\]
where \(S_t\) acts on the fixed-point set of \(\sigma\); see, for
example, \cite[Proposition 2.32]{KerberI}. Moreover,
\[
        C_G(\alpha)=C_{A_n}(\sigma).
\]

Let
\[
        E\cong C_p^r
\]
be the elementary abelian subgroup generated by the \(r\) disjoint
\(p\)-cycles occurring in \(\sigma\). We first bound \(r\) in terms of
\(p\) and \(k\).

If \(p\) is odd, then every \(p\)-cycle is even, and hence
\[
        E\leq C_{A_n}(\sigma).
\]
Therefore \(C_G(\alpha)\) contains all non-identity elements of \(E\),
and all of them have order \(p\). Thus
\[
        k\geq p^r-1,
\]
so
\[
        r\leq \log_p(k+1).
\]

If \(p=2\), then \(E\cap A_n\) consists of the products of an even
number of the transpositions occurring in \(\sigma\). Hence
\(E\cap A_n\) contains \(2^{r-1}-1\) non-identity elements, all of
order \(2\). Thus
\[
        k\geq 2^{r-1}-1,
\]
and therefore
\[
        r\leq 1+\log_2(k+1).
\]

Consequently, in all cases we have the uniform bound
\begin{equation}\label{lemmaalt:bound-r}
        r\leq R(p,k):=1+\log_p(k+1).
\end{equation}

We bound \(t\). If \(t\leq 2p\), then
\[
        n=rp+t
        \leq pR(p,k)+2p,
\]
and hence \(n\) is bounded in terms of \(p\) and \(k\).

Thus we may assume that \(t>2p\). Consider the elements of \(S_t\)
which are products of two disjoint \(p\)-cycles. Each such element
centralizes \(\sigma\), belongs to \(A_n\), and has order \(p\). The
number of such elements is
\begin{align*}
        \frac{1}{2}\binom{t}{p}(p-1)!
        \binom{t-p}{p}(p-1)!
        &=
        \frac{1}{2}
        \frac{t!}{p(t-p)!}
        \frac{(t-p)!}{p(t-2p)!}  \\
        &=
        \frac{t(t-1)\cdots(t-2p+1)}{2p^2}.
\end{align*}
Therefore
\[
        k
        \geq
        \frac{t(t-1)\cdots(t-2p+1)}{2p^2}
        \geq
        \frac{(t-2p+1)^{2p}}{2p^2}.
\]
It follows that
\[
        t\leq (2p^2k)^{1/(2p)}+2p-1.
\]
Using \eqref{lemmaalt:bound-r}, we obtain
\[
        n=rp+t
        \leq
        pR(p,k)+(2p^2k)^{1/(2p)}+2p-1.
\]
Thus \(n\) is bounded in terms of \(p\) and \(k\), as required.
\end{proof}

The preceding lemma will be used to handle the alternating groups in
the proof of the main theorem. Since the sporadic simple groups form a
finite family, the remaining work concerns finite simple groups of Lie
type. We recall the notation and the elementary number-theoretic
facts needed for that case.

Let \(\ell\) be a prime, and let \(\overline G\) be a simple adjoint
linear algebraic group over \(\overline{\mathbb F}_{\ell}\). Let \(F\)
be a Steinberg endomorphism of \(\overline G\), that is, an algebraic
endomorphism of \(\overline G\) with finite fixed-point group; see
\cite[Theorem 10.13]{st-endo}. We write
\[
        \overline G^F=\{g\in \overline G : F(g)=g\}.
\]
Then \(\overline G^F\) is a finite group of Lie type over the corresponding finite field $\mathbb{F}_{q}$ for an $\ell$-power $q$. The subgroup
\[
        G=O^{\ell'}(\overline G^F),
\]
generated by the \(\ell\)-elements of \(\overline G^F\), is a finite
simple group of Lie type whenever \(|G|\geq 60\). For definitions,
notation, and standard facts about finite groups of Lie type, we refer
to \cite{ca2,spst,gls3, lec-st} and \cite{st-endo}.

We shall use the following form of Zsigmondy's theorem.

\begin{theorem}[Zsigmondy {\cite{zsg}}]\label{zsgt}
Let \(a\) and \(b\) be coprime positive integers with \(a>b\). For every
positive integer \(n\), there exists a prime \(r\), called a primitive
prime divisor of \(a^n-b^n\), such that
\[
        r\mid a^n-b^n
        \qquad\text{and}\qquad
        r\nmid a^i-b^i
\]
for every positive integer \(i<n\), except in the following cases:
\begin{enumerate}
    \item \(n=1\) and \(a-b=1\);
    \item \(n=2\) and \(a+b\) is a power of \(2\);
    \item \(n=6\), \(a=2\), and \(b=1\).
\end{enumerate}
\end{theorem}

A group \(H\) is said to
satisfy the minimal condition on \(p\)-subgroups, or min-\(p\), if every
descending chain of \(p\)-subgroups of \(H\) stabilizes.

\begin{lemma}\label{finite-p-elements-min-p}
Let \(H\) be a locally finite group and let \(p\) be a prime. If \(H\)
contains only finitely many elements of order \(p\), then \(H\) satisfies
min-\(p\).
\end{lemma}

\begin{proof}
Let \(P\) be a \(p\)-subgroup of \(H\). Then \(P\) is a locally finite
\(p\)-group and contains only finitely many elements of order \(p\).
By Blackburn's theorem on locally finite \(p\)-groups, \(P\) is a
Chernikov group; see \cite[Theorem 4.1]{Blackburn}. Hence \(P\)
satisfies the minimal condition on subgroups. Therefore every
\(p\)-subgroup of \(H\) satisfies the minimal condition on subgroups,
and so \(H\) satisfies min-\(p\).
\end{proof}

\begin{theorem}[Kegel--Wehrfritz]\label{kegel-wehrfritz-min-p}
Let \(p\) be a prime, and let \(A\) be a finite \(p\)-group acting by
automorphisms on a locally finite group \(H\). If \(C_H(A)\) satisfies
min-\(p\), then \(H\) satisfies min-\(p\).
\end{theorem}

\begin{proof}
It follows from \cite[Corollary 3.2]{KegelWehrfritz}; see also
\cite[Proposition 3.6]{ShumyatskySurvey}.
\end{proof}

\begin{corollary}\label{cyclic-automorphism-min-p}
Let \(H\) be a locally finite group, and let \(\alpha\in\operatorname{Aut}(H)\)
have order \(p\), where \(p\) is prime. If \(C_H(\alpha)\) contains only
finitely many elements of order \(p\), then \(H\) satisfies min-\(p\).
\end{corollary}

\begin{proof}
By Lemma~\ref{finite-p-elements-min-p}, the fixed-point subgroup
\(C_H(\alpha)\) satisfies min-\(p\). Applying
Theorem~\ref{kegel-wehrfritz-min-p} with
\[
        A=\langle \alpha\rangle
\]
shows that \(H\) satisfies min-\(p\).
\end{proof}

\begin{theorem}[Belyaev--Borovik--Hartley--Shute--Thomas]
\label{classification-min-p-simple-locally-finite}
Let \(H\) be an infinite simple locally finite group, and let \(p\) be a
prime. Then \(H\) satisfies min-\(p\) if and only if \(H\) is a simple
group of Lie type over an infinite locally finite field \(K\) with
\[
        \operatorname{char} K\neq p.
\]
Equivalently, the infinite simple locally finite groups satisfying
min-\(p\) are precisely the simple locally finite groups of Lie type in
cross-characteristic \(p\).
\end{theorem}

\begin{proof}
This is the classification of simple locally finite groups satisfying
min-\(p\), due independently to Belyaev, Borovik, Hartley--Shute, and
Thomas; see \cite{BelyaevChevalley, BorovikEmbeddings, HartleyShute, Thomas}.

\end{proof}
\section{Proofs of the main results}

In this section we prove the main theorem and a locally finite
counterpart. We first treat elements of finite simple groups of Lie type.

\begin{lemma}\label{unitriangular-abelian-subgroup}
Let \(q=p^a\) and let \(n\geq 2\). Every element of the unitriangular
group \(\mathrm{U}_n(\mathbb F_q)\) is contained in an abelian subgroup of order
at least \(q\). More precisely, for every \(u\in \mathrm{U}_n(\mathbb F_q)\),
there exists an abelian subgroup \(A\leq \mathrm{U}_n(\mathbb F_q)\) such that
\[
        u\in A
        \qquad\text{and}\qquad
        |A|\geq q.
\]
\end{lemma}

\begin{proof}
Let \(E_{ij}\) denote the usual matrix units, and set
\[
        Z_0=\{\,I+\lambda E_{1n}:\lambda\in\mathbb F_q\,\}.
\]
We first observe that \(Z_0\) is contained in the center of
\(\mathrm{U}_n(\mathbb F_q)\). Indeed, if \(N\) is any strictly upper triangular
matrix, then
\[
        E_{1n}N=0
        \qquad\text{and}\qquad
        NE_{1n}=0.
\]
Therefore \(I+\lambda E_{1n}\) commutes with \(I+N\) for every
\(I+N\in \mathrm{U}_n(\mathbb F_q)\). Thus
\[
        Z_0\leq Z(\mathrm{U}_n(\mathbb F_q)).
\]
Moreover, \(Z_0\) is isomorphic to the additive group of
\(\mathbb F_q\), and hence
\[
        |Z_0|=q.
\]

Now let \(u\in \mathrm{U}_n(\mathbb F_q)\). Define
\[
        A=\langle u,Z_0\rangle.
\]
Since \(Z_0\) is central, the subgroup \(A\) is generated by the cyclic
group \(\langle u\rangle\) and the central subgroup \(Z_0\). Hence \(A\)
is abelian. It contains \(u\), and since \(Z_0\leq A\), we have
\[
        |A|\geq |Z_0|=q.
\]
This proves the claim.
\end{proof}
\begin{lemma}\label{unipotent}
Let \(k\in \mathbb N\). There exists a function
\[
        B:\mathbb N\to\mathbb N
\]
such that the following holds. Let \(G\) be a finite simple group of Lie
type, and let \(x\in G\) be a unipotent element of odd prime order
\(p\). If \(C_G(x)\) contains at most \(k\) elements of order \(p\), then
\[
        |G|\leq B(k).
\]
\end{lemma}

\begin{proof}
Let \(G\) be defined over \(\mathbb F_q\), where \(q=p^a\), and let
\(r\) be the Lie rank of \(G\). Since \(x\) is unipotent, \(p\) is the
defining characteristic. By Lemma~\ref{unitriangular-abelian-subgroup},
\(C_G(x)\) contains at least \(q-1\) elements of order \(p\). Hence
\[
        q-1\leq k,
        \qquad\text{so}\qquad
        q\leq k+1.
\]

The exceptional groups, including the Suzuki and Ree groups, have
bounded rank. We may therefore assume that \(G\) is classical. Let \(V\)
be a natural module for the corresponding quasisimple classical group
\(\widetilde G\), and put \(N=\dim V\). It is enough to bound \(N\).

Let \(\widetilde x\in\widetilde G\) be a preimage of \(x\). All Jordan
blocks of \(\widetilde x\) on \(V\) have size at most \(p\). Write
\[
        V=\bigoplus_{s=1}^{p} V_s,
\]
where \(V_s\) is the direct sum of the subspaces corresponding to the
Jordan blocks of size \(s\). If \(m_s\) is the number of these blocks,
then
\[
        N=\sum_{s=1}^{p}s m_s,
        \qquad
        \sum_{s=1}^{p}m_s\geq \frac{N}{p}.
\]
Hence, for some \(s\),
\[
        m_s\geq \frac{N}{p^2}.
\]

Fix such an \(s\). With respect to the decomposition of \(V_s\) into its
\(m_s\) Jordan blocks, the restriction of \(\widetilde x\) has the block
diagonal form
\[
        \operatorname{diag}(J_s,\ldots,J_s),
\]
where \(J_s\) is a unipotent Jordan block of size \(s\). In the linear
case, the following matrices commute with this block diagonal matrix:
\[
\begin{pmatrix}
I_s & \lambda_2 I_s & \lambda_3 I_s & \cdots & \lambda_{m_s} I_s\\
0   & I_s           & 0             & \cdots & 0\\
0   & 0             & I_s           & \cdots & 0\\
\vdots & \vdots     & \vdots        & \ddots & \vdots\\
0   & 0             & 0             & \cdots & I_s
\end{pmatrix},
\qquad \lambda_i\in\mathbb F_q.
\]
They form an elementary abelian \(p\)-subgroup of order \(q^{m_s-1}\).
This is also the \(\mathrm{GL}_{m_s}(q)\)-factor in the decomposition of
the centralizer of a unipotent element in the linear group; see
\cite[Theorem 2.3.9]{DeFranceschi2020}; see also \cite{FLO}.

For the unitary, symplectic and orthogonal groups, the centralizer
decomposition in \cite[Theorem 5.2.1]{DeFranceschi2020}; see also
\cite{FLO}, gives
\[
        C_{\widetilde G}(\widetilde x)=U\rtimes R,
\]
where \(U\) is a \(p\)-group and the component of \(R\) attached to the
Jordan blocks of size \(s\) is one of
\[
        \mathrm{U}_{m_s}(q),\qquad
        \mathrm{Sp}_{m_s}(q),\qquad
        \mathrm{O}_{m_s}^{\varepsilon}(q).
\]
In the symplectic and orthogonal cases the factor is symplectic or
orthogonal according to the parity of \(s\).

Assume first that this factor is \(\mathrm{U}_{m_s}(q)\), and put
\(a=\lfloor m_s/2\rfloor\). The stabilizer of a maximal totally
isotropic subspace of dimension \(a\) has a Levi subgroup containing
\(\mathrm{GL}_a(q^2)\). The linear construction above, applied over
\(\mathbb F_{q^2}\), gives an elementary abelian \(p\)-subgroup of order
\[
        (q^2)^{a-1}=q^{2a-2}\geq q^{m_s-3}.
\]

If the factor is \(\mathrm{Sp}_{m_s}(q)\), then \(m_s\) is even. The
stabilizer of a maximal totally isotropic subspace has Levi subgroup
\(\mathrm{GL}_{m_s/2}(q)\), and the linear construction gives an
elementary abelian \(p\)-subgroup of order
\[
        q^{m_s/2-1}.
\]

Finally suppose that the factor is \(\mathrm{O}_{m_s}^{\varepsilon}(q)\).
Since \(p\) is odd, we may work in the corresponding special orthogonal
group. If \(a\) is the Witt index, then \(a\geq m_s/2-1\). The stabilizer
of a totally singular subspace of dimension \(a\) has Levi subgroup
containing \(\mathrm{GL}_a(q)\). Thus it contains an elementary abelian
\(p\)-subgroup of order at least
\[
        q^{a-1}\geq q^{m_s/2-2}.
\]

Thus, in every classical case, \(C_{\widetilde G}(\widetilde x)\)
contains an elementary abelian \(p\)-subgroup of order at least
\(q^{m_s/4-2}\), after absorbing the cases \(m_s<8\). The kernel of the
map from \(\widetilde G\) to \(G\) has order prime to \(p\), apart from
finitely many low-dimensional cases. These cases may be absorbed into
the function \(B\). Hence \(C_G(x)\) contains at least
\[
        q^{m_s/4-2}-1
\]
elements of order \(p\).

This number is at most \(k\). Since \(q\leq k+1\), the integer \(m_s\) is
bounded in terms of \(k\). Since \(m_s\geq N/p^2\) and
\(p\leq q\leq k+1\), also \(N\), and hence the Lie rank, is bounded in
terms of \(k\). Therefore \(q\) and the rank are bounded in terms of
\(k\), and only finitely many groups can occur. This gives
\[
        |G|\leq B(k)
\]
for some function \(B:\mathbb N\to\mathbb N\).
\end{proof}
\begin{lemma}\label{semisimple}
Let \(k,e\in \mathbb N\). There exists a function
\[
        B:\mathbb N^2\to\mathbb N
\]
such that the following holds. Let \(G\) be a finite simple group of Lie
type, and let \(x\in G\) be a semisimple element of odd prime order \(p\).
If \(C_G(x)\) contains at most \(k\) elements of order \(p\), and
\[
        \exp C_G(x)\leq e,
\]
then
\[
        |G|\leq B(k,e).
\]
\end{lemma}

\begin{proof}
Let \(G\) be defined over \(\mathbb F_q\), where \(q=\ell^a\), and let
\(r\) be the Lie rank of \(G\). Since \(x\) is semisimple of order \(p\),
we have \(p\neq \ell\). The elements
\[
        x,x^2,\ldots,x^{p-1}
\]
belong to \(C_G(x)\) and have order \(p\). Hence
\[
        p-1\leq k,
        \qquad\text{so}\qquad
        p\leq k+1.
\]

The exceptional groups have bounded rank. We may assume that \(G\) is
classical. Let \(V\) be a natural module for the corresponding
quasisimple classical group \(\widetilde G\), and put \(N=\dim V\). It
suffices to bound \(N\).

Let \(\widetilde x\in\widetilde G\) be a preimage of \(x\). Since the
minimal polynomial of \(\widetilde x\) divides \(X^p-1\), every
irreducible \(\mathbb F_q\langle\widetilde x\rangle\)-module occurring
in \(V\) has dimension at most \(p-1\). Write
\[
        V=\bigoplus_{j=1}^{t} V_j,
\]
where each \(V_j\) is the sum of all copies of one irreducible module
\(W_j\), or of one dual pair in the form-preserving cases. Thus, in the
linear or self-dual case,
\[
        V_j\cong W_j^{\oplus m_j},
\]
while in the dual-pair case
\[
        V_j\cong (W_j\oplus W_j^\ast)^{\oplus m_j}.
\]
Here \(m_j\) is the multiplicity. Since \(t\leq p\) and the modules
\(W_j\) have dimension at most \(p-1\), an unbounded \(N\) gives an
unbounded \(m_j\) for some \(j\).

Fix such a \(j\). For linear groups, \cite[Theorem 2.3.6]{DeFranceschi2020};
see also \cite{FLO}, gives a factor
\[
        \mathrm{GL}_{m_j}(Q),
\]
where
\[
        Q=\operatorname{End}_{\mathbb F_q\langle\widetilde x\rangle}(W_j)
\]
is a finite extension of \(\mathbb F_q\) of degree at most \(p-1\). For
unitary, symplectic and orthogonal groups,
\cite[Theorem 4.2.1]{DeFranceschi2020}; see also \cite{FLO}, gives the
corresponding factor on the multiplicity space. According to the type of
the constituent, this factor is linear, unitary, symplectic or
orthogonal. Restrictions coming from determinant or spinor norm can only
decrease the \(p\)-rank by a bounded amount.

First consider a linear factor \(\mathrm{GL}_{m_j}(Q)\). Let \(f\) be
the multiplicative order of \(|Q|\) modulo \(p\). Then \(f\leq p-1\), and
\(\mathrm{GL}_f(Q)\) contains an element \(A\) of order \(p\). Placing
copies of \(A\) on disjoint \(f\)-dimensional blocks gives an elementary
abelian \(p\)-subgroup of rank at least
\[
        \left\lfloor\frac{m_j}{p-1}\right\rfloor .
\]

Now consider a unitary, symplectic or orthogonal factor. Take a maximal
totally isotropic, respectively totally singular, subspace. Its
stabilizer has a Levi subgroup containing a linear factor
\(\mathrm{GL}_a(E)\), where \(E\) is a finite field. In the unitary and
symplectic cases we may take
\[
        a=\left\lfloor\frac{m_j}{2}\right\rfloor ,
\]
and in the orthogonal case the Witt index satisfies
\[
        a\geq \frac{m_j}{2}-1.
\]
Applying the preceding linear construction to \(\mathrm{GL}_a(E)\) gives
an elementary abelian \(p\)-subgroup of rank at least
\[
        \left\lfloor\frac{m_j}{2(p-1)}\right\rfloor-1.
\]

It follows that \(C_{\widetilde G}(\widetilde x)\) contains an elementary
abelian \(p\)-subgroup of rank at least
\[
        \left\lfloor\frac{m_j}{2(p-1)}\right\rfloor-1,
\]
up to a bounded loss coming from the special and Omega restrictions.
Passing from \(\widetilde G\) to \(G\) factors by a central subgroup. In
the linear and unitary cases this centre is cyclic, and in the other
classical cases it has bounded order. Hence \(C_G(x)\) contains at least
\[
        p^{\lfloor m_j/(2(p-1))\rfloor-2}-1
\]
elements of order \(p\). Since this number is at most \(k\), the integer
\(m_j\) is bounded in terms of \(k\). Since \(p\leq k+1\) and \(t\leq p\),
\(N\), and hence the Lie rank, is bounded in terms of \(k\).

We next bound \(q\). With the rank bounded, only finitely many types of
semisimple centralizers occur. The connected centralizer of
\(\widetilde x\) contains a non-trivial \(F\)-stable torus \(\mathbf T\).
Put \(T=\mathbf T^F\). The rank of \(\mathbf T\) is positive and bounded,
and \(|T|\) is a product of cyclotomic factors \(\Phi_m(q)\), with \(m\)
in a finite set depending only on the root datum; see
\cite{ca2,MalleTesterman}.

Suppose that \(q\) is unbounded. If \(\ell\) is unbounded, then \(|T|\),
and hence \(\exp T\), is unbounded because \(T\) has bounded rank. Thus
we may take \(q=\ell^a\) with \(\ell\) fixed and \(a\to\infty\). Passing
to a subsequence, a fixed factor \(\Phi_m(q)\) divides \(|T|\). By
Theorem~\ref{zsgt}, for all large \(a\), the integer \(\ell^{am}-1\) has
a primitive prime divisor \(r_a\). Then
\[
        \operatorname{ord}_{r_a}(\ell)=am,
        \qquad
        \operatorname{ord}_{r_a}(q)=m,
\]
so \(r_a\mid \Phi_m(q)\). Since \(r_a\geq am+1\), the primes \(r_a\) are
unbounded. Thus \(\exp T\) is unbounded. Its image in \(C_G(x)\) is a
central quotient with kernel of bounded order, since the rank is bounded.
Hence the exponent of this image is also unbounded, contradicting
\[
        \exp C_G(x)\leq e.
\]

Therefore \(q\) is bounded in terms of \(k\) and \(e\). The rank is
already bounded, so only finitely many finite simple groups of Lie type
can occur. Hence
\[
        |G|\leq B(k,e).
\]
\end{proof}

Now, we are ready to prove the main theorem.

\begin{proof}[Proof of Theorem~\ref{mainthm}]
If \(G\) is cyclic of prime order, then the conclusion is immediate.
Thus we may assume that \(G\) is non-abelian.

First note that
\[
        x,x^2,\ldots,x^{p-1}
\]
all belong to \(C_G(x)\) and have order \(p\). Hence
\[
        p-1\leq k,
\]
so \(p\leq k+1\).

By the classification of finite simple groups, \(G\) is either an
alternating group, a sporadic group, or a finite simple group of Lie
type.

If \(G=A_m\) is an alternating group with \(m\geq 7\), then
Lemma~\ref{lemmaalt} gives
\[
        m\leq f_0(p,k)
\]
for some function \(f_0\). Since \(p\leq k+1\), the degree \(m\), and
therefore \(|A_m|\), is bounded in terms of \(k\). The alternating
groups of degree \(m<7\) form a finite family.

The sporadic simple groups also form a finite family. Hence their
orders are bounded by an absolute constant.

We are left with finite simple groups of Lie type. Let \(G\) be
defined over \(\mathbb F_q\), where \(q=\ell^a\). Since \(x\) has prime
order, it is unipotent if \(p=\ell\), and semisimple if \(p\neq \ell\).

\medskip

\noindent\textbf{Case 1: \(x\) is unipotent.}
In this case \(p=\ell\). By Lemma~\ref{unipotent}, the order of \(G\)
is bounded in terms of \(k\), and hence in terms of \(k\) and \(e\).

\medskip

\noindent\textbf{Case 2: \(x\) is semisimple.}
In this case \(p\neq \ell\). By Lemma~\ref{semisimple}, the order of
\(G\) is bounded in terms of \(k\) and \(e\).

Thus, in all cases, \(|G|\) is bounded in terms of \(k\) and \(e\).
This proves the theorem.
\end{proof}

The following result completes our discussion.

\begin{proposition}\label{locfinsimp}
Let \(G\) be a simple locally finite group, and let \(\alpha\) be an
automorphism of \(G\) of odd prime order \(p\). Suppose that
\(C_G(\alpha)\) contains finitely many elements of order \(p\), and that
\(\pi(C_G(\alpha))\) is also finite.
Then \(G\) is finite.
\end{proposition}

\begin{proof}
Suppose, for a contradiction, that \(G\) is infinite. Since
\(C_G(\alpha)\) contains only finitely many elements of order \(p\),
it satisfies min-\(p\). By Theorem~\ref{kegel-wehrfritz-min-p},
\(G\) satisfies min-\(p\).

By Theorem~\ref{classification-min-p-simple-locally-finite}, the
group \(G\) is a simple group of Lie type over an infinite locally
finite field \(K\) of characteristic \(\ell\neq p\).

Now Hartley's theorem applies. More precisely, by
\cite[Theorem C]{Hartley1992}, if an infinite simple group of Lie type
over a locally finite field of characteristic different from \(p\) admits
an automorphism of order \(p\), then this automorphism fixes elements
of infinitely many distinct prime orders. Therefore
\[
        |\pi(C_G(\alpha))|=\infty,
\]
contrary to the hypothesis.

Hence \(G\) is finite.
\end{proof}

\section*{Acknowledgements}

Part of this work was carried out while the second author held a temporary
professorship at Technische Universität Dresden and the first author was
visiting the Institut für Algebra. The first two authors thank Ellen Henke
and the Institut für Algebra for their hospitality and excellent research atmosphere.

The first author is supported by the ``National Group for Algebraic and
Geometric Structures and their Applications'' (GNSAGA -- INdAM), and by
the SDF Sustainability Decision Framework Research Project -- MISE decree
of 31/12/2021 (MIMIT Dipartimento per le politiche per le imprese --
Direzione generale per gli incentivi alle imprese) -- CUP:~B79J23000530005,
COR:~14019279, Lead Partner:~TD Group Italia Srl, Partner:~University of
Palermo.

\end{document}